\documentclass[11pt,reqno]{amsart}

\usepackage{amssymb}
\usepackage{amsmath,amsthm,amsfonts,amssymb,latexsym,mathrsfs,color,hyperref}
\usepackage{graphicx}
\usepackage{xspace}
\usepackage{multirow}
\usepackage{diagbox}
\usepackage{capt-of}
\usepackage{tikz}
\tikzset{
	level 1/.style = {sibling distance = 1.5cm},
	level 2/.style = {sibling distance = 0.8cm},
    level distance = 1.0 cm
}
\usetikzlibrary {decorations.pathmorphing}
\tikzstyle{snakeline} = [decorate, decoration={snake, amplitude=.4mm, segment length=2mm}]

\newtheorem{theorem}{Theorem}
\newtheorem{corollary}[theorem]{Corollary}
\newtheorem{proposition}[theorem]{Proposition}

\newtheorem{lemma}[theorem]{Lemma}

\newtheorem{example}[theorem]{Example}

\newcommand{\Des}{\operatorname{Des}}
\newcommand{\SYT}{\operatorname{SYT}}
\newcommand{\maj}{\operatorname{maj}}
\newcommand{\Hom}{H}
\newcommand{\G}{\mathcal G}

\newcommand{\id}{\operatorname{id}}
\newcommand{\qbinom}[2]{\genfrac{[}{]}{0pt}{}{#1}{#2}_{q}}

\newcommand{\des}{{\rm des\,}}
\newcommand{\asc}{{\rm asc\,}}
\newcommand{\msn}{{\mathcal S}_n}

\title{Eulerian insertion operators and an Eulerian form of the Pieri rule}
\author[S.-M.~Ma]{Shi-Mei Ma}
\address{School of Mathematics and Statistics, Shandong University of Technology, Zibo, Shandong 255000, P.R. China}
\email{shimeimapapers@163.com (S.-M. Ma)}
\begin{document}

\maketitle
\begin{abstract}
We study the operators obtained by inserting copies of a new largest
letter into multiset permutations.  Let \(\G_r\) denote the
operator which inserts \(r\) copies of a new largest letter.
After the change of variables
\(\delta=y-x\), \(u=x/y\), and \(E=u\partial_u\), we find that
\[
 \G_r=\frac{\delta^r}{r!}E(E+1)\cdots(E+r-1).
\]
Its generating series acts by a rational substitution, which yields a
composition law.

Our main result gives a common symmetric-function explanation for the ordinary and major-index operators. For $N\geq 0$, define
$\Phi_N(F_{N,S})=x^{|S|+1}y^{N-|S|}$.
We prove that multiplication by the complete homogeneous symmetric function $h_r$ becomes the ordinary insertion operator:
$\Phi_{N+r}(h_r f)=\G_r\Phi_N(f)$, where $f\in\mathrm{QSym}_N$.
There is a parallel specialization for the major index. A reverse finite principal specialization sends multiplication by
\(h_r\) to an operator \(Q_r\), which is a polynomial in the
\(q\)-shift \(\Theta_qf(t)=f(qt)\).
Thus the ordinary and major-index operators arise from the same multiplication operator $f\mapsto h_r f$. 
Since the functions \(h_r\) freely generate the ring of symmetric
functions, the assignment \(h_r\mapsto\G_r\) extends to an algebra
homomorphism. We determine the kernel of this homomorphism and the image of every
homogeneous component. The images of Schur functions satisfy the Littlewood--Richardson
multiplication identities, and the one-row case gives an Eulerian form of the Pieri
rule.  
\bigskip

\end{abstract}
\date{\today}
\tableofcontents
\section{Introduction}
Let $\msn$ be the set of all permutations of $[n]:=\{1,2,\ldots,n\}$.
As usual, we write $\pi=\pi(1)\pi(2)\cdots\pi(n)\in\msn$. For \(i\in[n-1]\),
we say that $i$ is a {\it descent} (resp.~{\it ascent})
if $\pi(i)>\pi(i+1)$ (resp.~$\pi(i)<\pi(i+1)$).
Let $\des(\pi)$ and $\asc(\pi)$ be the numbers of descents and ascents of $\pi$, respectively.
A bivariate version of the Eulerian polynomial over the symmetric group $\msn$ is given as follows:
$$A_n(x,y)=\sum_{\pi\in\msn}x^{\des(\pi)+1}y^{n-\des(\pi)}.$$ 
Carlitz and Scoville~\cite{CarlitzScoville} found that
\begin{equation*}
A_{n+1}(x,y)=xy\left(\partial_x+\partial_y\right)A_n(x,y),~A_1(x,y)=xy.
\end{equation*}
The complement map
\(\pi_i\mapsto n+1-\pi_i\) exchanges descents and ascents, and hence
\(A_n(x,y)=A_n(y,x)\). When $y=1$, the polynomial $A_n(x,y)$ reduces to the classical Eulerian polynomial $A_n(x)$.
Eulerian polynomials have long been studied through recurrences,
generating functions and differential operators. The reader is referred to~\cite{Carlitz1973,DHS2026,MaPan2023,Petersen,Tielker} for background and further details. 
The real-rootedness of multiset Eulerian polynomials is known~\cite{BHVW2011,Simion}. 
We focus instead on the operators which describe the effect of adjoining
a new letter with prescribed multiplicity.

Let \(\boldsymbol m=(m_1,\ldots,m_n)\) be a vector of nonnegative integers,
and let \(\mathfrak S_{\boldsymbol m}\) be the set of permutations of
$\{1^{m_1},2^{m_2},\ldots,n^{m_n}\}$.
Zero multiplicities are ignored.  Write \(N=m_1+\cdots+m_n\).  For
\(\pi=\pi_1\cdots\pi_N\), let $\Des(\pi)=\{i\in[N-1]:\pi_i>\pi_{i+1}\}$ and let
\(\des(\pi)=\#\Des(\pi)\). For a nonempty multiset $\boldsymbol m$, set
\begin{equation}\label{eq:homogeneous-enumerator}
 \Hom_{\boldsymbol m}(x,y)
 =\sum_{\pi\in\mathfrak S_{\boldsymbol m}}
   x^{\des(\pi)+1}y^{N-\des(\pi)}.
\end{equation}
For the empty multiset, set
$\Hom_{\varnothing}(x,y)=x$.
For a nonempty word
\(\pi=\pi_1\cdots\pi_N\), set
\(\pi_0=\pi_{N+1}=0\).  Its \(N+1\) gaps are
\[
 (\pi_0,\pi_1),(\pi_1,\pi_2),\ldots,(\pi_N,\pi_{N+1}).
\]
Label a gap by \(x\) if its left entry is larger than its right entry, and
by \(y\) otherwise.  The product of the gap labels is the weight of the word
in \eqref{eq:homogeneous-enumerator}.

Throughout this paper, we always set $D=\partial_x+\partial_y$.
Two differential operators arising in Eulerian enumeration were found to
commute \cite{MaQiYehYeh}.  Their commutativity is easy to verify but less easy
to explain. 
They are
\[
 \G_1=xyD,
 \qquad
 \G_2=xy^2D+\frac{x^2y^2}{2}D^2.
\]
The first operator inserts one new largest letter, and the second inserts two
copies of a new largest letter.  Nothing in these formulas explains why
the order of the two insertions should be irrelevant.
The two operators are the first two members of a single factorial sequence.  
To see this, consider inserting \(r\) copies of a new largest letter.
If the inserted letters form \(j\) nonempty blocks, their lengths form a
composition of \(r\), while the blocks occupy \(j\) gaps of the old word.
This gives
\begin{equation}\label{eq:intro-Gr}
 \G_r
 =y^r\sum_{j=1}^{r}
   \binom{r-1}{j-1}\frac{x^j}{j!}D^j,
 \qquad \G_0=\id.
\end{equation}

This paper is organized as follows.
In Section~\ref{section2}, we first prove that \(\G_r\) is exactly the operator for inserting
\(r\) copies of a new largest letter, and 
we then consider the generating series of $ \G_r$, a change of variables, and the composition of two insertions.
The central result of Section~\ref{section3} is
\[
 \Phi_{N+r}(h_rf)=\G_r\Phi_N(f).
\]
Thus, under the ordinary specialization, multiplication by \(h_r\)
becomes the insertion operator \(\G_r\). The assignment \(h_r\mapsto\G_r\) then extends
to the ring of symmetric functions and carries Schur multiplication to
identities among the insertion operators.  Section~\ref{section4} proves
the parallel identity
\[
 \Psi_{N+r,q}(h_rf;t)=Q_r\Psi_{N,q}(f;t),
\]
in which \(Q_r\) is a polynomial in the \(q\)-shift.
\section{Insertion copies of a new largest
letter and the Euler operator}\label{section2}
\subsection{Eulerian insertion operators and generating series}
\begin{proposition}\label{prop:insertion}
For every multiplicity vector \(\boldsymbol m\) and every \(r\geq1\), we have
\[
 \Hom_{(m_1,\ldots,m_n,r)}(x,y)
 =\G_r\Hom_{\boldsymbol m}(x,y).
\]
\end{proposition}

\begin{proof}
Recall that $\Hom_{\varnothing}(x,y)=x$.
Note that $\G_r(x)=xy^r= \Hom_{(r)}(x,y)$. We may therefore assume that the original word is nonempty.
Fix \(\pi\in\mathfrak S_{\boldsymbol m}\), and let \(L\) be larger
than every letter of \(\pi\).  After inserting \(r\) copies of \(L\), the new letters
form \(j\) nonempty maximal blocks for some \(1\leq j\leq r\).  From left to
right, let their lengths be
\(a_1,\ldots,a_j\).  Thus $a_1+\cdots+a_j=r$, where $a_i\geq1$.
There are \(\binom{r-1}{j-1}\) possible lists of block lengths.
The \(j\) blocks must occupy \(j\) distinct gaps of \(\pi\).
Suppose that \(\pi\) has weight \(x^ay^b\), so \(a\) gaps are labelled
\(x\) and \(b\) gaps are labelled \(y\).  Since
\(D=\partial_x+\partial_y\), we obtain
\[
\begin{aligned}
 \frac{D^j}{j!}(x^ay^b)
 &=\frac1{j!}\sum_{s=0}^j\binom{j}{s}
   \partial_x^s\partial_y^{j-s}(x^ay^b)\\
 &=\sum_{s=0}^j
   \binom{a}{s}\binom{b}{j-s}
   x^{a-s}y^{b-j+s}.
\end{aligned}
\]
The second equality follows from
\[
 \partial_x^s x^a=s!\binom{a}{s}x^{a-s},
 \qquad
 \partial_y^{j-s}y^b=(j-s)!\binom{b}{j-s}y^{b-j+s},
\]
together with $\binom{j}{s}s!(j-s)!=j!$.
For fixed \(s\), the coefficient
\(\binom{a}{s}\binom{b}{j-s}\) chooses \(s\) of the \(x\)-labelled gaps
and \(j-s\) of the \(y\)-labelled gaps, while
\(x^{a-s}y^{b-j+s}\) records the labels of the gaps not chosen.
Summing over \(s\) therefore counts every unordered set of \(j\)
distinct gaps exactly once.  Thus \(D^j/j!\) is precisely the operator
which chooses those \(j\) gaps.

Suppose that a block \(L^{a_i}\) is inserted in the gap
\((\pi_s,\pi_{s+1})\).  Since \(L\) is larger than every old letter,
the old adjacency is replaced by
\[
 \pi_s<L=\cdots=L>\pi_{s+1}.
\]
This description also holds for the two boundary gaps, since
\(\pi_0=\pi_{N+1}=0<L\).  The new chain contains \(a_i\)
non-descents and one descent just before $\pi_{s+1}$.  Its weight is therefore
\[
 xy^{a_i}.
\]
If the \(j\) blocks have lengths \(a_1,\ldots,a_j\), their combined
weight is
\[
 \prod_{i=1}^jxy^{a_i}
 =x^jy^{a_1+\cdots+a_j}
 =x^jy^r.
\]
As shown above, \(D^j/j!\) chooses the \(j\) occupied gaps and removes
their old labels from the weight.  Since there are
\(\binom{r-1}{j-1}\) compositions of \(r\) into \(j\) positive parts,
the contribution of all insertions with exactly \(j\) blocks is
\[
 y^r\binom{r-1}{j-1}\frac{x^j}{j!}D^j.
\]
Summing over \(j=1,\ldots,r\) gives \(\G_r\).
The construction is reversible.  In the resulting word, the maximal
runs of \(L\)'s are precisely the inserted blocks.  Deleting these runs
recovers \(\pi\); their positions recover the occupied gaps, and their
lengths, read from left to right, recover the composition
\((a_1,\ldots,a_j)\).  Hence every resulting multiset permutation is
counted exactly once.
\end{proof}

\begin{example}[Three equal letters]
Suppose that three copies of a new largest letter are inserted into a multiset permutation.  They form
one, two, or three maximal blocks.  The three contributions are
\[
 xy^3D,
 \qquad
 2x^2y^3\frac{D^2}{2!},
 \qquad
 x^3y^3\frac{D^3}{3!}.
\]
Thus
\[
 \G_3=xy^3D+x^2y^3D^2+\frac{x^3y^3}{6}D^3.
\]
The coefficients \(1,2,1\) count compositions of \(3\) with the specified
number of parts.
\end{example}

\begin{corollary}
For \(\boldsymbol m=(m_1,\ldots,m_n)\), we have $\Hom_{\boldsymbol m}(x,y)
 =\G_{m_n}\cdots\G_{m_1}(x)$.
\end{corollary}
\begin{proof}
Zero multiplicities may be omitted; equivalently, we use
\(\G_0=\id\).
The empty multiset has enumerator \(x\).  Starting from it, insert
\(m_1\) copies of \(1\), then \(m_2\) copies of \(2\), and continue in
increasing order of the letters.  At the \(i\)-th step, the letter \(i\)
is larger than every letter already present, so
Proposition~\ref{prop:insertion} shows that the current enumerator is
transformed by \(\G_{m_i}\).  After the \(n\)-th step, the resulting
multiset is $\{1^{m_1},2^{m_2},\ldots,n^{m_n}\}$.
Successive application of the operators therefore gives
$\Hom_{\boldsymbol m}(x,y)
 =\G_{m_n}\cdots\G_{m_1}(x)$.
\end{proof}

Note that $Dx=Dy=1$ and $D^jx=D^jy=0$ for $j\geq2$.
By~\eqref{eq:intro-Gr}, we see that
$\G_r(x)=\G_r(y)=xy^r$, where $r\geq1$.
Hence
\begin{align*}
 x+\sum_{r\geq1}\G_r(x)z^r
 &=x+\sum_{r\geq1}xy^rz^r
   =\frac{x}{1-yz},\\
 y+\sum_{r\geq1}\G_r(y)z^r
 &=y+\sum_{r\geq1}xy^rz^r
   =\frac{y(1-yz+xz)}{1-yz}.
\end{align*}
We define
\[
 \G(z)=\sum_{r\geq0}\G_rz^r.
\]
The next lemma shows that the same substitution acts on every polynomial.

\begin{lemma}[A substitution]\label{lem:substitution}
For every \(f\in\mathbb Q[x,y]\),
\begin{equation}\label{eq:substitution}
 \G(z)f(x,y)
 =f\left(
     \frac{x}{1-yz},
     \frac{y\bigl(1-yz+xz\bigr)}{1-yz}
   \right),
\end{equation}
where the right-hand side is expanded at \(z=0\).
\end{lemma}

\begin{proof}
Starting from \eqref{eq:intro-Gr} and changing the order of summation gives
\begin{align*}
 \G(z)
 &=\id+
   \sum_{r\geq1}\sum_{j=1}^r
   \binom{r-1}{j-1}\frac{x^jy^rz^r}{j!}D^j\\
 &=\id+
   \sum_{j\geq1}\frac{x^j}{j!}D^j
   \sum_{r\geq j}\binom{r-1}{j-1}(yz)^r\\
 &=\id+
   \sum_{j\geq1}\frac1{j!}
   \left(\frac{xyz}{1-yz}\right)^jD^j\\
 &=\sum_{j\geq0}\frac1{j!}
   \left(\frac{xyz}{1-yz}\right)^jD^j.
\end{align*}

Let \(\eta\) be an auxiliary indeterminate independent of \(x\) and
\(y\).  Since \(f\) is a polynomial, Taylor's formula gives
\[
 \sum_{j\geq0}\frac{\eta^j}{j!}D^jf(x,y)
 =f(x+\eta,y+\eta).
\]
Applying the specialization
\[
 \eta\longmapsto\frac{xyz}{1-yz}
\]
to this polynomial identity gives
\[
\begin{aligned}
 \G(z)f(x,y)
 &=
 f\left(
   x+\frac{xyz}{1-yz},
   y+\frac{xyz}{1-yz}
 \right)\\
 &=
 f\left(
   \frac{x}{1-yz},
   \frac{y(1-yz+xz)}{1-yz}
 \right),
\end{aligned}
\]
which proves \eqref{eq:substitution}.
\end{proof}

The generating substitution fixes \(y-x\), while homogeneity reduces
the two variables to their ratio.  Set
\[
 \delta=y-x,\qquad u=\frac{x}{y},\qquad E=u\partial_u.
\]

\begin{lemma}[Change of variables]\label{lem:coordinates}
We have
\[
 x=\frac{\delta u}{1-u},
 \qquad
 y=\frac{\delta}{1-u}.
\]
Under the substitution of Lemma~\ref{lem:substitution},
\[
 \delta\longmapsto\delta,
 \qquad
 u\longmapsto\frac{u}{1-\delta z}.
\]
Moreover, $\G_1=\delta E$
on \(\mathbb Q(\delta,u)\).  In particular, \(E\) commutes with
multiplication by \(\delta\).
\end{lemma}

\begin{proof}
Since \(x=uy\), we have $\delta=y-x=y(1-u)$.
Therefore $y=\frac{\delta}{1-u}$.
Multiplying by \(u\) gives
\[
 x=uy=\frac{\delta u}{1-u}.
\]
Under the substitution of Lemma~\ref{lem:substitution},
\[
\begin{aligned}
 y-x
 &\longmapsto
 \frac{y(1-\delta z)-x}{1-yz}
 =\delta,\\
 \frac{x}{y}
 &\longmapsto
 \frac{x}{y(1-\delta z)}
 =\frac{u}{1-\delta z}.
\end{aligned}
\]
Note that $(\partial_x+\partial_y)\delta=0$ and $(\partial_x+\partial_y)u
 =\frac1y-\frac{x}{y^2}
 =\frac{\delta}{y^2}$.
Thus, for \(F=F(\delta,u)\),
\[
 \G_1F
 =xy(\partial_x+\partial_y)F
 =\delta uF_u
 =\delta EF.
\]
Hence \(\G_1=\delta E\).  Since \(\delta\) is independent of \(u\), we have \(E(\delta)=0\).
Hence, for every \(F\in\mathbb Q(\delta,u)\),
\[
 E(\delta F)=\delta EF,
\]
so \(E\) commutes with multiplication by \(\delta\).
\end{proof}

\subsection{The factorial formula and the composition law}
\hspace*{\parindent}

We are now ready to present the first main result of this paper.
\begin{theorem}[The factorial formula]\label{thm:factorial}
In the coordinates \((\delta,u)\), the following identity holds as a
formal power series in \(z\):
\begin{equation*}
 \G(z)
 =\exp\bigl(-\log(1-\delta z)E\bigr).
\end{equation*}
Consequently, for $r\geq 1$, 
\begin{equation}\label{eq:rising-factorial}
 \G_r
 =\frac{\delta^r}{r!}E(E+1)\cdots(E+r-1).
\end{equation}
Equivalently, on \(\mathbb Q[x,y]\),
\begin{equation}\label{eq:operator-factorization}
 r!\G_r
 =\prod_{j=0}^{r-1}(xyD+j(y-x)),
\end{equation}
where \(y-x\) acts by multiplication.
In particular,  the operators
\(\G_r\) commute pairwise.
\end{theorem}
\begin{proof}
Let \(f\in\mathbb Q[x,y]\).  After substituting
\[
 x=\frac{\delta u}{1-u},
 \qquad
 y=\frac{\delta}{1-u},
\]
the resulting expression belongs to
\(\mathbb Q[\delta,u,(1-u)^{-1}]\).  Since \(1-u\) has constant term
\(1\), it is invertible in the formal power-series ring
\(\mathbb Q[\delta][[u]]\).  Hence there are unique polynomials
\(c_m(\delta)\in\mathbb Q[\delta]\) such that
\[
  f\left(\frac{\delta u}{1-u},
          \frac{\delta}{1-u}\right)
  =
  \sum_{m\geq0}c_m(\delta)u^m.
\]
 This is the formal expansion at
\(u=0\). 
The change of variables is birational: its inverse is
\[
  \delta=y-x,
  \qquad
  u=\frac{x}{y}.
\]
Thus \(\mathbb Q(x,y)=\mathbb Q(\delta,u)\), and the map which sends
\(f(x,y)\) to its formal expansion above is injective.
By Lemma~\ref{lem:coordinates}, the generating operator \(\G(z)\)
fixes \(\delta\) and replaces \(u\) by \(u/(1-\delta z)\).  Therefore,
\[
  \G(z)f
  =
  \sum_{m\geq0}
  c_m(\delta)
  \left(\frac{u}{1-\delta z}\right)^m.
\]

Set $a(z)=-\log(1-\delta z)$. Since \(a(z)\) depends only on \(\delta\), Lemma~\ref{lem:coordinates}
shows that multiplication by \(a(z)\) commutes with \(E\).
Then the formal operator series
\[
  \exp\bigl(a(z)E\bigr)
  =
  \sum_{n\geq0}\frac{a(z)^nE^n}{n!}
\]
is well defined: for each fixed power of \(z\), only finitely many
values of \(n\) contribute.  Since \(E(u^m)=mu^m\), we have
\[
\begin{aligned}
  \exp\bigl(a(z)E\bigr)u^m
  &=
  \sum_{n\geq0}
    \frac{a(z)^nm^n}{n!}u^m\\
  &=
  \exp\bigl(ma(z)\bigr)u^m\\
  &=
  (1-\delta z)^{-m}u^m.
\end{aligned}
\]
Consequently,
\[
\begin{aligned}
  \exp\bigl(-\log(1-\delta z)E\bigr)f
  &=
  \sum_{m\geq0}
  c_m(\delta)(1-\delta z)^{-m}u^m\\
  &=
  \sum_{m\geq0}
  c_m(\delta)
  \left(\frac{u}{1-\delta z}\right)^m\\
  &=
  \G(z)f.
\end{aligned}
\]
The injectivity noted above now proves the operator identity on
\(\mathbb Q[x,y]\).
Taking the coefficient of \(z^r\) in the preceding identity gives
\[
\begin{aligned}
  \G_r f
  &=
  \delta^r\sum_{m\geq0}
  \binom{m+r-1}{r}c_m(\delta)u^m\\
  &=
  \frac{\delta^r}{r!}
  E(E+1)\cdots(E+r-1)
  \sum_{m\geq0}c_m(\delta)u^m.
\end{aligned}
\]
Hence
\[
  \G_r
  =
  \frac{\delta^r}{r!}E(E+1)\cdots(E+r-1).
\]

Finally, Lemma~\ref{lem:coordinates} gives
\[
 \delta E=\G_1=xyD.
\]
Therefore $\delta(E+j)=xyD+j(y-x)$.
Multiplying these identities for \(j=0,1,\ldots,r-1\) gives
\eqref{eq:operator-factorization}.  Since \(E\) commutes with
multiplication by \(\delta\), formula \eqref{eq:rising-factorial} shows
that all the operators \(\G_r\) commute.
\end{proof}

The normalized coefficients \(\delta^{-r}\G_r\) are the divided rising
factorials in \(E\).  This is the classical factorial sequence of finite
operator calculus \cite{RomanRota,RotaKahanerOdlyzko}.
The point here is that the operators for inserting copies of a new largest letter are precisely
these divided rising factorials.
The explicit substitution also determines the composition of two insertions.

\begin{theorem}[Composition law]
For independent indeterminates \(z\) and \(v\),
\begin{equation*}
 \G(z)\G(v)=\G(z+v-\delta zv).
\end{equation*}
\end{theorem}

\begin{proof}
According to Theorem~\ref{thm:factorial}, we have $\G(z)=\exp\bigl(-\log(1-\delta z)E\bigr)$.

By Lemma~\ref{lem:coordinates}, \(\G(v)\) fixes \(\delta\) and sends
\[
  u\longmapsto\frac{u}{1-\delta v}.
\]
Applying \(\G(z)\) afterward sends \(u\) to
\[
  \frac{u}{(1-\delta z)(1-\delta v)}.
\]
Since $(1-\delta z)(1-\delta v)
  =
  1-\delta(z+v-\delta zv)$,
the composite substitution is the substitution defining
\(\G(z+v-\delta zv)\). The series \(z+v-\delta zv\) has zero constant term, so substitution
into \(\G(w)=\sum_{r\geq0}\G_rw^r\) is well defined.
\end{proof}

Since $E(u^m)=mu^m$,
we have $E^n(u^m)=m^n u^m$.
Hence
\[
\begin{aligned}
\exp\bigl(\log(1-\delta z)E\bigr)u^m
&=
\sum_{n\geq0}
\frac{\bigl(\log(1-\delta z)\bigr)^n}{n!}
E^n(u^m)\\
&=
\sum_{n\geq0}
\frac{\bigl(m\log(1-\delta z)\bigr)^n}{n!}u^m\\
&=
\exp\bigl(m\log(1-\delta z)\bigr)u^m\\
&=
(1-\delta z)^m u^m.
\end{aligned}
\]
By the binomial theorem,
\[
 (1-\delta z)^m
 =
 \sum_{r\geq0}
 (-1)^r\binom{m}{r}\delta^r z^r.
\]
On the other hand,
$E(E-1)\cdots(E-r+1)u^m
 =
 m(m-1)\cdots(m-r+1)u^m$.
Hence
\[
 \exp\bigl(\log(1-\delta z)E\bigr)u^m
 =
 \sum_{r\geq0}
 (-1)^r\frac{\delta^r}{r!}
 E(E-1)\cdots(E-r+1)u^m z^r.
\]
Since the two exponents commute,
\[
  \G(z)\exp\bigl(\log(1-\delta z)E\bigr)=\id.
\]
Hence the second exponential is the formal inverse of \(\G(z)\).
We use the notation
\[
  (1-\delta z)^E
  :=
  \exp\bigl(\log(1-\delta z)E\bigr).
\]
Both sides act coefficientwise on \(\mathbb Q[\delta][[u]]\), so the
calculation on the monomials \(u^m\) determines the operator on the
whole space.
\begin{corollary}
We have
$$\G(z)^{-1}
 =(1-\delta z)^E=
 \sum_{r\geq0}
 (-1)^r\frac{\delta^r}{r!}
 E(E-1)\cdots(E-r+1)z^r.$$
\end{corollary}

\begin{corollary}\label{cor:product-law}
For \(a,b\geq0\),
\[
 \G_a\G_b
 =\sum_{j=0}^{\min(a,b)}(-1)^j
   \frac{(a+b-j)!}{(a-j)!(b-j)!j!}\,
   (y-x)^j\G_{a+b-j},
\]
where \((y-x)^j\) acts by multiplication.
\end{corollary}
\begin{proof}
Recall that
$$\G(z+v-\delta zv)
 =\sum_{r\geq0}\G_r(z+v-\delta zv)^r,$$
 where $\delta$ acts as the multiplication by $y-x$. Multiplication by \(y-x\) commutes with every \(\G_r\).
 Take the coefficient of \(z^av^b\) in $\G(z)\G(v)=\G(z+v-\delta zv)$.
In the term $\G_r(z+v-\delta zv)^r$,
the monomial \(z^av^b\) is obtained by choosing, for some \(j\),
\(a-j\) factors \(z\), \(b-j\) factors \(v\), and \(j\) factors
\(-\delta zv\).  Necessarily \(r=a+b-j\), and the number of such
choices is
\[
 \frac{(a+b-j)!}{(a-j)!(b-j)!j!}.
\]
Summing over \(0\leq j\leq\min(a,b)\) gives the result.
\end{proof}

\section{A common symmetric-function source}\label{section3}
\subsection{The ordinary specialization}
\hspace*{\parindent}

The Eulerian insertion operator records the number of descents, whereas a
fundamental quasisymmetric function records their positions. We begin by 
relating these two levels of information. Let \(\mathrm{QSym}_N\) denote the degree-\(N\) homogeneous component
of the ring of quasisymmetric functions.
Let \(z_1,z_2,\ldots\) be commuting variables.  For \(N\geq1\) and
\(S\subseteq[N-1]\), define
\[
 F_{N,S}
 =\sum_{\substack{1\leq i_1\leq\cdots\leq i_N\\
                   j\in S\Rightarrow i_j<i_{j+1}}}
   z_{i_1}\cdots z_{i_N},
\]
and put \(F_{0,\varnothing}=1\).  These fundamental quasisymmetric
functions form a basis of \(\mathrm{QSym}_N\).  For \(r\geq0\), let \(h_r\) denote the complete homogeneous symmetric
function of degree \(r\), defined by
\[
 h_r=\sum_{1\leq i_1\leq\cdots\leq i_r}z_{i_1}\cdots z_{i_r},
 \qquad h_0=1.
\]
Thus \(h_r=F_{r,\varnothing}\). Let $\Phi_N:\mathrm{QSym}_N\longrightarrow\mathbb Q[x,y]$ be
the linear map determined on the fundamental basis by $\Phi_N(F_{N,S})=x^{|S|+1}y^{N-|S|}$, 
and we set $\Phi_0(1)=x$. Thus \(\Phi_N\) forgets the positions of the descents and retains only their number. 
In particular, it sends a descent-set enumerator to its homogeneous Eulerian enumerator.

For \(f\in\mathrm{QSym}_N\), we write $f(1^{k+1})=f(\underbrace{1,\ldots,1}_{k+1},0,0,\ldots)$.
For example, let
\[
f=F_{3,\{1\}}
=\sum_{i_1<i_2\le i_3} z_{i_1}z_{i_2}z_{i_3}.
\]
Then $f(1^3)=F_{3,\{1\}}(1,1,1,0,\ldots)
=4$,
since there are exactly four triples
\[
(1,2,2),\qquad
(1,2,3),\qquad
(1,3,3),\qquad
(2,3,3)
\]
satisfying $1\le i_1<i_2\le i_3\le 3$.
More generally, for a sequence
\[
 1\leq i_1\leq\cdots\leq i_N\leq k+1,
 \qquad
 j\in S\Longrightarrow i_j<i_{j+1},
\]
define
\[
 b_\ell
 =
 i_\ell-\#\{j\in S:j<\ell\}
 \qquad(1\leq\ell\leq N).
\]
Then $1\leq b_1\leq\cdots\leq b_N\leq k+1-|S|$.
This correspondence is reversible, since
\[
 i_\ell
 =
 b_\ell+\#\{j\in S:j<\ell\}.
\]
Hence
\[
 F_{N,S}(1^{k+1})
 =
 \binom{N+k-|S|}{N},
\]
where the binomial coefficient is understood to be zero when \(k<|S|\).

The next lemma is a basic tool for transferring multiplication by complete homogeneous symmetric functions to the Eulerian side.
\begin{lemma}[Finite principal specialization]\label{lem:Phi-principal}
If \(f\in\mathrm{QSym}_N\), then
\[
 \Phi_N(f)=\delta^{N+1}\sum_{k\geq0}f(1^{k+1})u^{k+1}.
\]
The identity is understood in \(\mathbb Q[\delta][[u]]\).
\end{lemma}
\begin{proof}
By linearity, it is enough to take \(f=F_{N,S}\).
Note that
\[
 \sum_{k\geq0}F_{N,S}(1^{k+1})u^k
 =\sum_{k\geq0}\binom{N+k-|S|}{N}u^k
 =\frac{u^{|S|}}{(1-u)^{N+1}}.
\]
Multiplying by \(u\delta^{N+1}\), and using
\(\delta=y(1-u)\) and \(u=x/y\), gives
\[
 u\delta^{N+1}\frac{u^{|S|}}{(1-u)^{N+1}}
 =x^{|S|+1}y^{N-|S|}
 =\Phi_N(F_{N,S}).
\]
Since the functions \(F_{N,S}\) form a basis of
\(\mathrm{QSym}_N\), the identity follows for every
\(f\) by linearity.
\end{proof}

\begin{theorem}[Multiplication by $h_r$]\label{thm:pieri-specialization}
For $N,r\geq0$ and $f\in\mathrm{QSym}_N$, we have
$$\Phi_{N+r}(h_rf)=\G_r\Phi_N(f).$$
\end{theorem}
\begin{proof}
The case \(r=0\) follows from \(h_0=1\) and \(\G_0=\id\). Assume henceforth that \(r\geq1\).
The proof is based on the fact that multiplication by \(h_r\) and the
operator \(\G_r\) produce the same factor on each term in the principal
specialization expansion of \(f\).

The term indexed by \(k\) in Lemma~\ref{lem:Phi-principal} is
$\delta^{N+1}f(1^{k+1})u^{k+1}$.
We first examine what happens on the symmetric-function side.
Evaluation at a fixed alphabet preserves products, so
\[
 (h_rf)(1^{k+1})
 =
 h_r(1^{k+1})f(1^{k+1})
\]
Now \(h_r(1^{k+1})\) counts weakly increasing sequences
$1\leq i_1\leq\cdots\leq i_r\leq k+1$,
or, equivalently, multisets of size \(r\) chosen from \(k+1\) symbols.
Thus $h_r(1^{k+1})=\binom{k+r}{r}$.
Applying Lemma~\ref{lem:Phi-principal} in degree \(N+r\), we therefore obtain
\[
 \Phi_{N+r}(h_rf)
 =
 \delta^{N+r+1}
 \sum_{k\geq0}
 \binom{k+r}{r}f(1^{k+1}) u^{k+1}.
\]

On the other hand, Theorem~\ref{thm:factorial} gives 
\[
 \G_r
 =
 \frac{\delta^r}{r!}
 E(E+1)\cdots(E+r-1),
 \qquad
 E=u\partial_u.
\]
The monomial \(u^{k+1}\) is an eigenvector of \(E\), since
$E(u^{k+1})=(k+1)u^{k+1}$.
Consequently, for $j\geq0$,
\[
 (E+j)u^{k+1}
 =
 (k+1+j)u^{k+1},
\]
and hence
\[
 \begin{aligned}
 \frac1{r!}
 E(E+1)\cdots(E+r-1)u^{k+1}
 &=
 \frac{(k+1)(k+2)\cdots(k+r)}{r!}\,u^{k+1}\\
 &=
 \binom{k+r}{r}u^{k+1}.
 \end{aligned}
\]
Therefore
\[
 \G_r\bigl(\delta^{N+1}f(1^{k+1}) u^{k+1}\bigr)
 =
 \delta^{N+r+1}
 \binom{k+r}{r}f(1^{k+1})u^{k+1}.
\]
Since \(\G_r\) is a polynomial in \(E\) with coefficient
\(\delta^r\), it acts coefficientwise on
\(\mathbb Q[\delta][[u]]\).
Applying \(\G_r\) term by term to
\[
 \Phi_N(f)
 =
 \delta^{N+1}\sum_{k\geq0}f(1^{k+1}) u^{k+1},
\]
gives
\[
 \G_r\Phi_N(f)
 =
 \delta^{N+r+1}
 \sum_{k\geq0}
 \binom{k+r}{r}f(1^{k+1}) u^{k+1}.
\]
This is exactly the expression obtained above for
\(\Phi_{N+r}(h_rf)\).  Hence
$\Phi_{N+r}(h_rf)=\G_r\Phi_N(f)$,
as required. This completes the proof.
\end{proof}

The proof also explains why a rising factorial occurs in
Theorem~\ref{thm:factorial}.  Under finite principal specialization,
multiplication by \(h_r\) multiplies the \(k\)-th coefficient by
\[
 h_r(1^{k+1})=\binom{k+r}{r}.
\]
On the Eulerian side, the normalized operator
\[
 \delta^{-r}\G_r
 =
 \frac1{r!}E(E+1)\cdots(E+r-1)
\]
has the same eigenvalue on \(u^{k+1}\):
\[
 \delta^{-r}\G_r(u^{k+1})
 =
 \binom{k+r}{r}u^{k+1}.
\]
Thus the rising factorial in \(E\) is precisely the operator form of the
finite principal specialization
\[
 h_r(1^{k+1})=\binom{k+r}{r}.
\]

\begin{corollary}
\label{cor:symmetric-multiset}
Let \(\boldsymbol m=(m_1,\ldots,m_n)\) have total size \(N\), and set
$h_{\boldsymbol m}=h_{m_1}\cdots h_{m_n}$.
Then
\begin{equation}\label{eq:fundamental-multiset}
 h_{\boldsymbol m}
 =\sum_{\pi\in\mathfrak S_{\boldsymbol m}}
   F_{N,\Des(\pi)},
 \qquad
 \Phi_N(h_{\boldsymbol m})
 =\Hom_{\boldsymbol m}(x,y).
\end{equation}
Moreover, for every permutation \(\sigma\) of \([n]\),
$\Hom_{\boldsymbol m}(x,y)
 =\G_{m_{\sigma(n)}}\cdots
  \G_{m_{\sigma(1)}}(x)$.
\end{corollary}

\begin{proof}
Recall that
\[
 h_{m_a}
 =
 \sum_{1\leq i_{a,1}\leq\cdots\leq i_{a,m_a}}
 z_{i_{a,1}}\cdots z_{i_{a,m_a}}.
\]
Thus a term in the expansion of \(h_{\boldsymbol m}\) is specified by
one weakly increasing index sequence for each letter \(a\).  Attach the
label \(a\) to every entry of the sequence for that letter, and sort all
\(N\) labeled entries first by their indices and, when the indices are
equal, by their labels.  Write
\[
 i_1\leq\cdots\leq i_N
\]
for the resulting indices and
\(\pi=\pi_1\cdots\pi_N\) for the resulting labels.  Then
\(\pi\in\mathfrak S_{\boldsymbol m}\).  Moreover, if
\(\pi_j>\pi_{j+1}\), then \(i_j=i_{j+1}\) is impossible, since equal
indices were ordered by increasing label.  Hence
\[
 j\in\Des(\pi)\quad\Longrightarrow\quad i_j<i_{j+1}.
\]
The resulting monomial therefore occurs in \(F_{N,\Des(\pi)}\).
Conversely, let \(\pi\in\mathfrak S_{\boldsymbol m}\) and let
\(i_1\leq\cdots\leq i_N\) be strict at every descent of \(\pi\).
On each block of equal indices, the corresponding letters are weakly
increasing.  Thus the pairs \((i_j,\pi_j)\) are already in
lexicographic order.  Grouping them according to their second
coordinates recovers one weakly increasing sequence for each letter.
The two constructions are inverse and preserve the monomial weight,
which proves the first identity in
\eqref{eq:fundamental-multiset}.
Applying \(\Phi_N\) gives
\[
 \Phi_N(h_{\boldsymbol m})
 =\sum_{\pi\in\mathfrak S_{\boldsymbol m}}
   x^{|\Des(\pi)|+1}y^{N-|\Des(\pi)|}
 =\Hom_{\boldsymbol m}(x,y).
\]

For the empty word, set $\Des(\varnothing)=\varnothing$.
Recall that \(\Phi_0(1)=x\).
Taking \(N=0\), \(f=1\), and \(r=m_{\sigma(1)}\) in
Theorem~\ref{thm:pieri-specialization} gives
\[
 \G_{m_{\sigma(1)}}(x)
 =
 \Phi_{m_{\sigma(1)}}(h_{m_{\sigma(1)}}).
\]
Iterating the same identity yields
\[
 \G_{m_{\sigma(n)}}\cdots\G_{m_{\sigma(1)}}(x)
 =
 \Phi_N\!\left(
 h_{m_{\sigma(n)}}\cdots h_{m_{\sigma(1)}}
 \right).
\]
Since multiplication of symmetric functions is commutative,
\[
 h_{m_{\sigma(n)}}\cdots h_{m_{\sigma(1)}}
 =
 h_{m_1}\cdots h_{m_n}
 =
 h_{\boldsymbol m}.
\]
Therefore, it follows from
\(\Phi_N(h_{\boldsymbol m})=\Hom_{\boldsymbol m}(x,y)\) that $\Hom_{\boldsymbol m}(x,y)
 =\G_{m_{\sigma(n)}}\cdots
  \G_{m_{\sigma(1)}}(x)$.
\end{proof}

Once Theorem~\ref{thm:pieri-specialization} is established, the
commutativity of the functions \(h_{m_i}\) explains the independence of
the order of insertion for the Eulerian enumerators.  The stronger
commutativity of the operators on all of \(\mathbb Q[x,y]\) was proved
in Theorem~\ref{thm:factorial}.
It is therefore natural to ask whether the
correspondence $h_r\longmapsto \G_r$
extends from the complete homogeneous functions to the whole ring of
symmetric functions. The next subsection shows that it does.

\subsection{The symmetric-function homomorphism}
\hspace*{\parindent}

Let \(\Lambda\) denote the ring of symmetric functions over
\(\mathbb Q\).  The complete homogeneous functions freely generate $\Lambda=\mathbb Q[h_1,h_2,\ldots]$.
Since the operators \(\G_r\) commute, the same polynomial can be
evaluated at \(\G_1,\G_2,\ldots\).  This defines a unique unital algebra
homomorphism
\begin{equation}\label{rho}
 \rho:\Lambda\longrightarrow
 \operatorname{End}_{\mathbb Q}\bigl(\mathbb Q[x,y]\bigr),
 \qquad
 \rho(h_r)=\G_r,
\end{equation}
where multiplication in the endomorphism algebra is composition.
Concretely, $\rho(h_{r_1}\cdots h_{r_\ell})
 =\G_{r_1}\cdots\G_{r_\ell}$.
The order on the right is immaterial because the insertion operators
commute.

A partition is a weakly decreasing sequence
\(\lambda=(\lambda_1,\ldots,\lambda_\ell)\) of positive integers. Its size
is \(|\lambda|=\lambda_1+\cdots+\lambda_\ell\), and its Young diagram has
\(\lambda_i\) cells in row \(i\). Let \(s_\lambda\) be the Schur function of shape \(\lambda\).
We now apply \(\rho\) to Schur functions.  For the partition
\(\lambda=(\lambda_1,\ldots,\lambda_\ell)\), the Jacobi--Trudi identity
gives
\[
 s_\lambda
 =\det\bigl(h_{\lambda_i-i+j}\bigr)_{1\leq i,j\leq\ell},
\]
where \(h_0=1\) and \(h_m=0\) for \(m<0\), see~\cite[Chapter~7]{StanleyEC2}.  We therefore define
\[
 \G_\lambda=\rho(s_\lambda)
 =\det\bigl(\G_{\lambda_i-i+j}\bigr)_{1\leq i,j\leq\ell},
\]
with \(\G_0=\id\) and \(\G_m=0\) for \(m<0\).  The determinant is unambiguous because all its entries commute.
In particular,
$\G_{(r)}=\G_r$,
since \(s_{(r)}=h_r\). 
For example, since \(h_0=1\), the Jacobi--Trudi identity gives
\[
 s_{(2,1)}
 =
 \det
 \begin{pmatrix}
  h_2 & h_3\\
  h_0 & h_1
 \end{pmatrix}
 =
 h_2h_1-h_3.
\]
Applying \(\rho\) therefore gives
$\G_{(2,1)}=\G_2\G_1-\G_3$.

For partitions \(\lambda\) and \(\mu\), let
\(c_{\lambda\mu}^{\nu}\) be the Littlewood--Richardson coefficients,
defined by
\[
 s_\lambda s_\mu
 =\sum_{\nu}c_{\lambda\mu}^{\nu}s_\nu.
\]
Let \(\Lambda_m\) denote the homogeneous component of degree \(m\).
The next theorem transfers the classical multiplication
identities for Schur functions to the operators \(\G_\lambda\).
\begin{theorem}[The symmetric-function action]
\label{thm:schur-multiplication}
Let \(g\in\Lambda_m\) and \(f\in\mathrm{QSym}_N\).  Since every symmetric function is quasisymmetric, we have
$\Lambda_m\subseteq\mathrm{QSym}_m$.
Thus \(gf\in\mathrm{QSym}_{N+m}\).
Then we have
\[
 \Phi_{N+m}(gf)=\rho(g)\Phi_N(f).
\]
In particular, for every partition \(\lambda\),
$\Phi_{N+|\lambda|}(s_\lambda f)
 =\G_\lambda\Phi_N(f)$.
Moreover, for all partitions \(\lambda\) and \(\mu\), we have
\[
 \G_\lambda\G_\mu
 =\sum_{\nu\vdash|\lambda|+|\mu|}
   c_{\lambda\mu}^{\nu}\G_\nu.
\]
\end{theorem}
\begin{proof}
If \(m=0\), then \(g\) is a scalar and the identity is immediate.
Assume \(m\geq1\).
For a partition
\(\alpha=(\alpha_1,\ldots,\alpha_\ell)\vdash m\), repeated application
of Theorem~\ref{thm:pieri-specialization} gives
\[
\begin{aligned}
 \Phi_{N+m}(h_\alpha f)
 &=
 \Phi_{N+m}
 \bigl(h_{\alpha_1}\cdots h_{\alpha_\ell}f\bigr)=
 \G_{\alpha_1}\cdots\G_{\alpha_\ell}\Phi_N(f)=
 \rho(h_\alpha)\Phi_N(f).
\end{aligned}
\]
The functions \(h_\alpha\), with \(\alpha\vdash m\), form a basis of
\(\Lambda_m\).  Since both sides depend linearly on \(g\), it follows
that
\[
 \Phi_{N+m}(gf)=\rho(g)\Phi_N(f)
\]
for every \(g\in\Lambda_m\).  Taking \(g=s_\lambda\) gives the second
identity.
Finally, since \(\rho\) is an algebra homomorphism,
\[
\begin{aligned}
 \G_\lambda\G_\mu
 &=\rho(s_\lambda)\rho(s_\mu)=\rho(s_\lambda s_\mu)=\rho\left(
   \sum_{\nu}c_{\lambda\mu}^{\nu}s_\nu
   \right)=\sum_{\nu}c_{\lambda\mu}^{\nu}\G_\nu.
\end{aligned}
\]
Only partitions \(\nu\) of size
\(|\lambda|+|\mu|\) can occur, which proves the last identity.
\end{proof}
The Littlewood--Richardson coefficients are classical.  The new point
is that the homomorphism \(\rho\) realizes these multiplication
identities by explicit operators generated by the Eulerian insertion operators.  

\begin{example}
The Pieri identity $h_2h_1=s_{(3)}+s_{(2,1)}$
becomes, under the homomorphism \(\rho\),
\[
 \G_2\G_1=\G_3+\G_{(2,1)}.
\]
We may also read this identity after applying the operators to
\(x=\Phi_0(1)\).  By Corollary~\ref{cor:symmetric-multiset},
\[
 \G_2\G_1(x)
 =\Hom_{(2,1)}(x,y)
 =xy^3+2x^2y^2.
\]
Indeed, the three words of content \((2,1)\) are $112,~121$ and $211$.
The first has no descent, while the other two have one descent.  Since $\G_3(x)=xy^3$,
the operator identity gives $\G_{(2,1)}(x)=2x^2y^2$.
At the level of the Eulerian enumerator, the two Schur terms account
for the term with no descent and the two terms with one descent,
respectively.
\end{example}

Take \(\lambda=(a)\) and \(\mu=(b)\), the one-row Pieri rule gives
\[
 s_{(a)}s_{(b)}
 =
 \sum_{j=0}^{\min(a,b)}
 s_{(a+b-j,j)}.
\]
Applying the algebra homomorphism \(\rho\) gives the following result.
\begin{corollary}[Two one-row operators]
For \(a,b\geq0\),
\[
 \G_a\G_b
 =
 \sum_{j=0}^{\min(a,b)}
 \G_{(a+b-j,j)}.
\]
Zero parts are omitted, and \(\G_\varnothing=\id\).
\end{corollary}

For \(m\geq1\), let $p_m=\sum_{i\geq1}z_i^m$
be the \(m\)-th power-sum symmetric function.  
For a partition \(\lambda=(\lambda_1,\ldots,\lambda_\ell)\), write
\(p_\lambda=p_{\lambda_1}\cdots p_{\lambda_\ell}\) and $\ell(\lambda)=\ell$.
The homomorphism \(\rho\) defined by~\eqref{rho} is not injective.  The next result determines
its kernel and the image of every homogeneous component.
\begin{proposition}[Power sums and the kernel of \(\rho\)]
\label{prop:kernel-rho}
We have
$\rho(p_m)=\delta^mE=\delta^{m-1}\G_1$.
Consequently, if \(\lambda\vdash n\), then
$\rho(p_\lambda)=\delta^nE^{\ell(\lambda)}$.
Moreover, $\ker\rho
  =
  \left\langle
    p_i p_{j+1}-p_{i+1}p_j:\ 1\leq i<j
  \right\rangle$.
In particular, for \(n\geq1\),
\[
  \rho(\Lambda_n)
  =
  \delta^n\operatorname{span}_{\mathbb Q}
  \{E,E^2,\ldots,E^n\},~
  \dim\ker(\rho|_{\Lambda_n})=p(n)-n,
\]
where \(p(n)\) is the number of partitions of \(n\).
\end{proposition}

\begin{proof}
Let \(w\) be a formal variable.  The generating series for $h_r$ is
\[
  \sum_{r\geq0}h_rw^r
  =
  \exp\left(
    \sum_{m\geq1}\frac{p_mw^m}{m}
  \right).
\]
Applying \(\rho\) coefficientwise in \(w\), and using the factorial
formula, gives
\[
\begin{aligned}
  \exp\left(
    \sum_{m\geq1}\frac{\rho(p_m)w^m}{m}
  \right)
  &=
  \sum_{r\geq0}\G_rw^r\\
  &=
  \exp\bigl(-\log(1-\delta w)E\bigr)\\
  &=
  \exp\left(
    \sum_{m\geq1}\frac{\delta^mEw^m}{m}
  \right).
\end{aligned}
\]
Both sides have constant term \(\operatorname{id}\), so their formal
logarithms are defined. Comparing coefficients of \(w^m\) gives
\(\rho(p_m)=\delta^mE\).
Since \(\G_1=\delta E\), we get $\rho(p_m)=\delta^{m-1}\G_1$.
Multiplicativity now gives $\rho(p_\lambda)
  =
  \delta^{|\lambda|}E^{\ell(\lambda)}$.

Let \(I\) be the ideal generated by
$p_i p_{j+1}-p_{i+1}p_j$, where $1\leq i<j$.
The preceding formula shows immediately that \(I\subseteq\ker\rho\).
For \(2\leq a\leq b\), one of the defining relations gives
\[
  p_ap_b
  \equiv
  p_{a-1}p_{b+1}
  \pmod I.
\]
Choose one largest part of a partition \(\lambda\vdash n\).  The preceding
relation transfers one unit from any other part greater than \(1\) to the
chosen largest part.  Repeating this operation gives
\[
  p_\lambda
  \equiv
  p_1^{\ell(\lambda)-1}
  p_{n-\ell(\lambda)+1}
  \pmod I.
\]
Hence the degree-\(n\) component of \(\Lambda/I\) is spanned by
\[
  p_{n},\quad
  p_1p_{n-1},\quad
  p_1^2p_{n-2},\quad\ldots,\quad
  p_1^{n}.
\]
Their images are $\delta^nE,\quad
  \delta^nE^2,\quad\ldots,\quad
  \delta^nE^n$.
These operators are linearly independent.
Indeed,
\[
 \delta^nE^\ell
 =\delta^{n-\ell}\G_1^\ell
 =\delta^{n-\ell}(xy)^\ell D^\ell
   +\text{terms of differential order less than \(\ell\)}.
\]
Thus, in a nonzero linear combination of these operators, the term of
largest differential order cannot be cancelled.
It follows that the map induced by \(\rho\) on the degree-\(n\)
component of \(\Lambda/I\) is injective. Hence
\[
  (\ker\rho)\cap\Lambda_n=I\cap\Lambda_n
\]
for every \(n\geq1\); equality in degree zero is immediate.
If \(g\in\Lambda_n\), then \(\rho(g)\) sends every homogeneous
polynomial of degree \(d\) either to zero or to a homogeneous polynomial
of degree \(d+n\).  Thus contributions from distinct homogeneous
components cannot cancel, and \(\ker\rho\) is homogeneous.  Since \(I\)
is homogeneous and $(\ker\rho)\cap\Lambda_n=I\cap\Lambda_n$
for every \(n\), we conclude that \(\ker\rho=I\).

Finally, the power sums \(p_\lambda\), with \(\lambda\vdash n\), form a
basis of \(\Lambda_n\), and their images depend only on
\(\ell(\lambda)\).  Every value \(1,\ldots,n\) occurs as a partition
length, so
\[
  \rho(\Lambda_n)
  =
  \operatorname{span}_{\mathbb Q}
  \left\{
    \delta^{\,n-\ell}\G_1^\ell:
    1\leq\ell\leq n
  \right\}.
\]
Equivalently, in the coordinates \((\delta,u)\),
$\rho(\Lambda_n)
  =
  \delta^n\operatorname{span}_{\mathbb Q}
  \{E,E^2,\ldots,E^n\}$.
Thus \(\dim\rho(\Lambda_n)=n\), and since
\(\dim\Lambda_n=p(n)\), we obtain $\dim\ker(\rho|_{\Lambda_n})=p(n)-n$.
\end{proof}

Since \(p(1)=1\), \(p(2)=2\), and \(p(3)=3\), there are no nonzero
relations in degrees less than \(4\).  In degree \(4\), the kernel is
one-dimensional.  A direct calculation gives
\[
  p_1p_3-p_2^2
  =
  s_{(3,1)}-3s_{(2,2)}+s_{(2,1,1)}.
\]
Therefore
\[
  \G_{(3,1)}
  -3\G_{(2,2)}
  +\G_{(2,1,1)}
  =0.
\]

\subsection{An Eulerian form of the Pieri rule}
\hspace*{\parindent}

We now transport the classical Pieri rule through the ordinary
specialization.
 Let \(\lambda\vdash N\).  A standard Young tableau of shape
\(\lambda\) is a filling of its diagram with \(1,2,\ldots,N\), strictly increasing
along rows and columns.  For \(T\in\SYT(\lambda)\), set
\[
 \Des(T)
 =
 \{j\in[N-1]:j+1\text{ lies in a lower row than }j\},
\]
and define
\[
 Y_\lambda(x,y)
 =
 \sum_{T\in\SYT(\lambda)}
 x^{|\Des(T)|+1}y^{N-|\Des(T)|}.
\]

The classical fundamental expansion of a Schur function is
\[
 s_\lambda
 =
 \sum_{T\in\SYT(\lambda)}F_{N,\Des(T)};
\]
see \cite{Gessel1984,StanleyEC2}.  
Since \(\Phi_N\) is
linear and $\Phi_N(F_{N,S})=x^{|S|+1}y^{N-|S|}$,
applying \(\Phi_N\) term by term gives
\[
\begin{aligned}
 \Phi_N(s_\lambda)
 &=
 \sum_{T\in\SYT(\lambda)}
 \Phi_N\bigl(F_{N,\Des(T)}\bigr)\\
 &=
 \sum_{T\in\SYT(\lambda)}
 x^{|\Des(T)|+1}y^{N-|\Des(T)|}\\
 &=
 Y_\lambda(x,y).
\end{aligned}
\]
Thus \(Y_\lambda(x,y)\) is the ordinary Eulerian specialization of
\(s_\lambda\). We now apply \(\Phi\) to the classical Pieri rule.
Recall also that a horizontal \(r\)-strip consists of \(r\) cells, no
two in the same column. The empty skew diagram is regarded as a horizontal \(0\)-strip.
The classical Pieri rule (see~\cite{StanleyEC2}) states that
\[
 h_rs_\lambda
 =
 \sum_{\substack{\nu\supseteq\lambda\\
                  \nu/\lambda\text{ is a horizontal }r\text{-strip}}}
 s_\nu.
\]

\begin{theorem}[An Eulerian form of the Pieri rule]
\label{thm:eulerian-pieri}
Let \(\lambda\vdash N\) and \(r\geq0\).  Then
\[
 \G_rY_\lambda(x,y)
 =
 \sum_{\substack{\nu\supseteq\lambda\\
                  \nu/\lambda\text{ is a horizontal }r\text{-strip}}}
 Y_\nu(x,y).
\]
\end{theorem}

\begin{proof}
Using \(\Phi_N(s_\lambda)=Y_\lambda\), Theorem~\ref{thm:pieri-specialization} and the classical Pieri rule, we obtain
\[
\begin{aligned}
 \G_rY_\lambda
 &=
 \G_r\Phi_N(s_\lambda)\\
 &=
 \Phi_{N+r}(h_rs_\lambda)\\
 &=
 \sum_{\substack{\nu\supseteq\lambda\\
                  \nu/\lambda\text{ is a horizontal }r\text{-strip}}}
 \Phi_{N+r}(s_\nu)\\
 &=
 \sum_{\substack{\nu\supseteq\lambda\\
                  \nu/\lambda\text{ is a horizontal }r\text{-strip}}}
 Y_\nu.
\end{aligned}
\]
\end{proof}

Thus multiplication by \(h_r\), which adds a horizontal strip on the
Schur-function side, becomes the action of the 
operator \(\G_r\) on the Eulerian tableau enumerators.

\begin{example}
Let \(\lambda=(2,1)\) and \(r=2\).  The classical Pieri rule gives
\[
 h_2s_{(2,1)}
 =
 s_{(4,1)}+s_{(3,2)}+s_{(3,1,1)}+s_{(2,2,1)}.
\]
These are exactly the partitions \(\nu\) for which
\(\nu/(2,1)\) is a horizontal \(2\)-strip.  The partition
\((2,1,1,1)\) is excluded because its two new cells lie in the same
column.
Theorem~\ref{thm:eulerian-pieri} therefore gives
\[
 \G_2Y_{(2,1)}
 =
 Y_{(4,1)}+Y_{(3,2)}+Y_{(3,1,1)}+Y_{(2,2,1)}.
\]
Since
\[
 Y_{(2,1)}=2x^2y^2
 \qquad\text{and}\qquad
 \G_2=xy^2D+\frac{x^2y^2}{2}D^2,
\]
direct differentiation gives
$\G_2Y_{(2,1)}
 =
 6x^2y^4+12x^3y^3+2x^4y^2$.
On the other hand, direct enumeration of the standard tableaux gives
\[
\begin{aligned}
 Y_{(4,1)}&=4x^2y^4,~Y_{(3,2)}=2x^2y^4+3x^3y^3,\\
 Y_{(3,1,1)}&=6x^3y^3,~Y_{(2,2,1)}=3x^3y^3+2x^4y^2.
\end{aligned}
\]
Their sum is $6x^2y^4+12x^3y^3+2x^4y^2$,
as predicted by Theorem~\ref{thm:eulerian-pieri}.
\end{example}

\section{The major-index specialization and the $q$-shift}\label{section4}
Throughout this section, \(q\) is an indeterminate, and all identities are understood in \(\mathbb{Q}(q)[[t]]\).
The map \(\Phi_N\) records only the number of descents.  To record the
sum of their positions, we use the reverse finite alphabet
$q^k,q^{k-1},\ldots,q,1$.
For a word \(\pi\), define
\[
 \maj(\pi)=\sum_{j\in\Des(\pi)}j,
 \qquad
 A_{\boldsymbol m}(t,q)
 =\sum_{\pi\in\mathfrak S_{\boldsymbol m}}
   t^{\des(\pi)}q^{\maj(\pi)}.
\]
For the empty multiset, set \(A_\varnothing(t,q)=1\).
We also write
\[
 (t;q)_s=\prod_{j=0}^{s-1}(1-tq^j).
\]
Empty products are understood to be \(1\).
For \(f\in\mathrm{QSym}_N\), set
\[
 \Psi_{N,q}(f;t)
 =
 \sum_{k\geq0}
 f(q^k,q^{k-1},\ldots,q,1,0,0,\ldots)t^k.
\]

\begin{lemma}[Reverse finite principal specialization]
\label{lem:q-principal}
For \(N\geq1\) and \(S\subseteq[N-1]\), we have
\[
 \Psi_{N,q}(F_{N,S};t)
 =
 \frac{t^{|S|}q^{\sum_{j\in S}j}}{(t;q)_{N+1}}.
\]
For \(N=0\), one has
\[
 \Psi_{0,q}(1;t)=\frac1{1-t}.
\]
\end{lemma}

\begin{proof}
Fix \(k\geq0\). Index the entries of the reverse alphabet by
\[
  z_{i+1}=q^{k-i},
  \qquad 0\leq i\leq k.
\]
Thus a monomial is indexed by
\[
 0\leq i_1\leq\cdots\leq i_N\leq k,
 \qquad
 j\in S\Longrightarrow i_j<i_{j+1},
\]
and has \(q\)-weight $q^{Nk-(i_1+\cdots+i_N)}$.
Introduce the nonnegative gaps
\[
 a_0=i_1,\qquad
 a_j=i_{j+1}-i_j-\mathbf1_{\{j\in S\}}
 \quad(1\leq j<N),\qquad
 a_N=k-i_N,
\]
where \(\mathbf1_{\{j\in S\}}\) is \(1\) if \(j\in S\) and \(0\)
otherwise.
Conversely, the indices are recovered from
\[
 i_\ell
 =
 a_0+\sum_{j<\ell}
 \bigl(a_j+\mathbf1_{\{j\in S\}}\bigr).
\]
For fixed \(k\), this is a bijection with the nonnegative tuples
satisfying
\[
 a_0+\cdots+a_N=k-|S|.
\]
Equivalently, as \(k\) varies, the variables
\(a_0,\ldots,a_N\) range independently over all nonnegative integers.  The gap equations give
$k=|S|+a_0+\cdots+a_N$
and
\[
 Nk-(i_1+\cdots+i_N)
 =
 \sum_{j\in S}j+a_1+2a_2+\cdots+Na_N.
\]
Therefore
\[
\begin{aligned}
 \Psi_{N,q}(F_{N,S};t)
 &=
 \sum_{a_0,\ldots,a_N\geq0}
 t^{|S|+a_0+\cdots+a_N}
 q^{\sum_{j\in S}j+a_1+2a_2+\cdots+Na_N}\\
 &=
 \frac{t^{|S|}q^{\sum_{j\in S}j}}
 {(1-t)(1-tq)\cdots(1-tq^N)}.
\end{aligned}
\]
The case \(N=0\) is the geometric series.
\end{proof}

Combining Lemma~\ref{lem:q-principal} with
Corollary~\ref{cor:symmetric-multiset} gives
\[
 \Psi_{N,q}(h_{\boldsymbol m};t)
 =
 \frac{A_{\boldsymbol m}(t,q)}{(t;q)_{N+1}}.
\]
Indeed, the term indexed by a word \(\pi\) contributes
\(t^{\des(\pi)}q^{\maj(\pi)}\). 

For \(k,r\geq0\), write
\[
 \qbinom{k+r}{r}
 =
 \prod_{\ell=1}^{r}
 \frac{1-q^{k+\ell}}{1-q^\ell},
\]
where empty products are understood to be \(1\).
Let \(\Theta_q\) denote the \(q\)-shift,
\[
 \Theta_qf(t)=f(qt),
\]
which acts on \(\mathbb Q(q)[[t]]\).
Define
\[
 Q_r=\prod_{i=1}^{r}\frac{1-q^i\Theta_q}{1-q^i},
 \qquad Q_0=\id.
\]

\begin{lemma}[The Gaussian coefficient]\label{lem:q-binomial}
For \(k,r\geq0\),
\[
 h_r(q^k,q^{k-1},\ldots,1)
 =\qbinom{k+r}{r},
 \qquad
 Q_rt^k=\qbinom{k+r}{r}t^k.
\]
\end{lemma}

\begin{proof}
The assertion is immediate if \(k=0\) or \(r=0\).  Assume from now on
that \(k,r\geq1\).
Since \(h_r\) is symmetric, we may use the alphabet
\(1,q,\ldots,q^k\).  Separate the monomials according to whether the last variable
\(q^k\) is absent or occurs at least once.  Removing one occurrence
of \(q^k\) in the latter case gives
\[
 h_r(1,q,\ldots,q^k)
 =
 h_r(1,q,\ldots,q^{k-1})
 +q^k h_{r-1}(1,q,\ldots,q^k).
\]
The Gaussian coefficients satisfy the same recurrence,
\[
 \qbinom{k+r}{r}
 =
 \qbinom{k+r-1}{r}
 +q^k\qbinom{k+r-1}{r-1}.
\]
The boundary values agree, so induction on
\(k+r\) proves the first identity.
For the second, \(\Theta_qt^k=q^kt^k\), and hence
\[
\begin{aligned}
 Q_rt^k
 &=
 \prod_{i=1}^{r}
 \frac{1-q^{k+i}}{1-q^i}\,t^k=
 \qbinom{k+r}{r}t^k.
\end{aligned}
\]
\end{proof}

The preceding lemma shows that multiplication by \(h_r\) under the
reverse finite specialization and the operator \(Q_r\) have the same
multiplier on the term \(t^k\).

Let
\[
 D_qf(t)=\frac{f(t)-f(qt)}{(1-q)t},
 \qquad
tD_q=\frac{I-\Theta_q}{1-q}.
\]
Thus \(tD_q(t^k)=[k]_qt^k\).
\begin{proposition}[A \(q\)-factorial form of \(Q_r\)]
For every \(r\geq0\),
\begin{equation}\label{eq:q-factorial-Euler}
 Q_r
 =
 \frac{1}{[r]_q!}
 \prod_{\ell=1}^r
 \bigl([\ell]_q+q^\ell tD_q\bigr),
\end{equation}
where $[r]_q!=[1]_q[2]_q\cdots[r]_q$.
\end{proposition}

\begin{proof}
Since $\Theta_q=I-(1-q)tD_q$,
we have
\[
\begin{aligned}
 1-q^\ell\Theta_q
 &=
 1-q^\ell+q^\ell(1-q)tD_q=
 (1-q)\bigl([\ell]_q+q^\ell tD_q\bigr).
\end{aligned}
\]
As \(1-q^\ell=(1-q)[\ell]_q\), it follows that
\[
 \frac{1-q^\ell\Theta_q}{1-q^\ell}
 =
 \frac{[\ell]_q+q^\ell tD_q}{[\ell]_q}.
\]
Multiplying over \(\ell=1,\ldots,r\) proves
\eqref{eq:q-factorial-Euler}.
\end{proof}
Thus the ordinary and \(q\)-operators have parallel factorial forms:
\[
 \delta^{-r}\G_r
 =
 \frac{E(E+1)\cdots(E+r-1)}{r!},
 \qquad
 Q_r
 =
 \frac{\prod_{\ell=1}^r([\ell]_q+q^\ell tD_q)}
      {[r]_q!}.
\]

\begin{theorem}[The \(q\)-specialization]\label{thm:q-factorial}
For \(N,r\geq0\) and \(f\in\mathrm{QSym}_N\), we have
\[
 \Psi_{N+r,q}(h_rf;t)
 =
 Q_r\Psi_{N,q}(f;t).
\]
The operators \(Q_r\) commute pairwise. 
\end{theorem}

\begin{proof}
We compare the coefficient of \(t^k\).  Since finite evaluation
preserves products, Lemma~\ref{lem:q-binomial} gives
\[
\begin{aligned}
\bigl[t^k\bigr]\Psi_{N+r,q}(h_rf;t)
 &=
 h_r(q^k,q^{k-1},\ldots,1)
 f(q^k,q^{k-1},\ldots,1)\\
 &=
 \qbinom{k+r}{r}
 [t^k]\Psi_{N,q}(f;t)\\
 &=
\bigl[t^k\bigr]Q_r\Psi_{N,q}(f;t).
\end{aligned}
\]
The coefficients agree for every \(k\), proving the first identity.
Each \(Q_r\) is a polynomial in the single operator \(\Theta_q\);
hence the operators \(Q_r\) commute.
\end{proof}

Starting with $\Psi_{0,q}(1;t)=\frac1{1-t}$,
and applying Theorem~\ref{thm:q-factorial} successively, we obtain
\[
 Q_{m_1}\cdots Q_{m_n}\frac1{1-t}= \frac{A_{\boldsymbol m}(t,q)}{(t;q)_{N+1}}.
\]
Since $$\frac1{1-t}=\sum_{k\geq0}t^k,~
 Q_{m_i}t^k
 =
 \qbinom{k+m_i}{m_i}t^k,$$
applying the operators term by term gives the following expression
\begin{equation}\label{eq:q-factorial-series}
 Q_{m_1}\cdots Q_{m_n}\frac1{1-t}
 =\frac{A_{\boldsymbol m}(t,q)}{(t;q)_{N+1}}
 =\sum_{k\geq0}
 \prod_{i=1}^n
 \qbinom{k+m_i}{m_i}t^k.
\end{equation}
The final equality in~\eqref{eq:q-factorial-series} is MacMahon's multiset major-index identity~\cite{MacMahon,Tielker}.

\begin{proposition}
\label{prop:ordinary-limit}
For \(r\geq0\), let $ E_t=t\partial_t$ and
\[
 B_r
 =
 \frac{(E_t+1)(E_t+2)\cdots(E_t+r)}{r!},
\]
where an empty product is \(1\).  Then, for every
\(F(t)\in\mathbb Q[[t]]\), we have $Q_rF(t)\longrightarrow B_rF(t)$
coefficientwise as \(q\to1\).  Moreover, in the coordinates
\((\delta,u)\), for every polynomial \(g(u)\),
$\G_r\bigl(ug(u)\bigr)
 =
 \delta^ru\,B_rg(u)$,
where \(B_r\) in the last identity acts in the variable \(u\).
\end{proposition}

\begin{proof}
For every \(k\geq0\),
\[
 Q_rt^k
 =
 \qbinom{k+r}{r}t^k
 \longrightarrow
 \binom{k+r}{r}t^k
 =
 B_rt^k.
\]
Indeed,
\[
  \frac{1-q^{k+i}}{1-q^i}
  \longrightarrow
  \frac{k+i}{i},
\]
and hence
\[
  \qbinom{k+r}{r}
  \longrightarrow
  \prod_{i=1}^r\frac{k+i}{i}
  =
  \binom{k+r}{r}.
\]
Thus, if \(F(t)=\sum_{k\geq0}a_kt^k\), the coefficient of \(t^k\)
in \(Q_rF(t)\) converges to the coefficient of \(t^k\) in
\(B_rF(t)\).
The second identity is immediate when \(r=0\).  For \(r\geq1\), write
\(E=u\partial_u\).  Since
$(E+j)(ug)=u(E+j+1)g$,
the factorial formula gives
\[
\begin{aligned}
 \G_r(ug)
 &=
 \frac{\delta^r}{r!}
 E(E+1)\cdots(E+r-1)(ug)\\
 &=
 \delta^ru\,
 \frac{(E+1)(E+2)\cdots(E+r)}{r!}g\\
 &=
 \delta^ru\,B_rg,
\end{aligned}
\]
as desired. This completes the proof.
\end{proof}
Proposition~\ref{prop:ordinary-limit} completes the comparison between the two operator realizations.
Under \(\Phi_N\), multiplication by \(h_r\) becomes the differential
operator \(\G_r\); under \(\Psi_{N,q}\), it becomes the \(q\)-shift
operator \(Q_r\).  As \(q\to1\), the same binomial multiplier appears in both operator
realizations.  The factor \(u\) accounts for the shift from \(t^k\)
to \(u^{k+1}\), while \(\delta^r\) records the increase in homogeneous
degree under ordinary insertion.  In conclusion, the two operator families arise
from one operation: multiplication by \(h_r\).

\end{document}